\documentclass[11pt, a4paper]{amsart}

\usepackage[T1]{fontenc}
\usepackage{lmodern}
\usepackage{amssymb, amsmath, amsthm, enumerate}

\theoremstyle{plain}
\newtheorem{thm}[subsubsection]{Theorem}
\newtheorem{prop}[subsubsection]{Proposition}
\newtheorem{lem}[subsubsection]{Lemma}

\theoremstyle{remark}
\newtheorem{rem}[subsubsection]{Remark}
\newtheorem{exam}[subsubsection]{Example}

\newtheorem*{prob}{Problem}

\title[]{Estimates for Weierstrass division in ultradifferentiable classes}

\author{Vincent Thilliez}

\address{Laboratoire Paul Painlev\'e, Math\'ematiques - B\^atiment M2\\
Universit\'e Lille 1\\
F-59655 Villeneuve d'Ascq Cedex, France}

\email{thilliez@math.univ-lille1.fr}


\begin{document}

\begin{abstract} 
We study the Weierstrass division theorem for function germs in strongly non-quas\-ian\-alytic Denjoy-Carleman classes $\mathcal{C}_M$. For suitable divisors $P(x,t)=x^d+a_1(t)x^{d-1}+\cdots+a_d(t)$ with real-analytic coefficients $a_j$, we show that the quotient and the remainder can be chosen of class $\mathcal{C}_{M^\sigma}$, where $M^\sigma=((M_j)^\sigma)_{j\geq 0}$ and $\sigma$ is a certain {\L}ojasiewicz exponent related to the geometry of the roots of $P$ and verifying $1\leq \sigma\leq d$. We provide various examples for which $\sigma$ is optimal, in particular strictly less than $d$, which sharpens earlier results of Bronshtein and of Chaumat-Chollet. 
\end{abstract}

\maketitle

\section*{Introduction} 
Consider a polynomial $P$ of degree $d$ given by 
\begin{equation*}
P(x,t)=x^d+a_1(t)x^{d-1}+\cdots+a_d(t)
\end{equation*}
where $a_1,\ldots,a_d$ are real-analytic function germs at the origin in $\mathbb{R}^m$ such that $a_j(0)=0$ for $j=1,\ldots,d$. The classic Weierstrass division theorem states that for any real-analytic function germ $f$ at the origin in $ \mathbb{R}\times\mathbb{R}^m $, we have a unique division formula 
\begin{equation*}
f(x,t)=P(x,t)q(x,t)+ \sum_{j=0}^{d-1}r_j(t)x^j,
\end{equation*}
where $q$ and $r_0,\ldots,r_{d-1}$ are real-analytic function germs at the origin in $ \mathbb{R}\times\mathbb{R}^m $ and $\mathbb{R}^m $, respectively. The famous Malgrange-Mather division theorem \cite{Ma, Mat, Tou} consists in a similar statement in the $\mathcal{C}^\infty$ setting: if $f$ and $a_1,\ldots,a_d$ are $\mathcal{C}^\infty$ function germs, the division formula still holds with $\mathcal{C}^\infty$ function germs $q$ and $r_0,\ldots,r_{d-1}$, which are no longer unique. 

In view of these two fundamental results, it is natural to ask what kind of division results can be obtained for classes of functions which are ``between analytic and $\mathcal{C}^\infty$'', namely Denjoy-Carleman classes $\mathcal{C}_M$ associated with a sufficiently regular sequence $M$ of real numbers (in particular, the stability of $\mathcal{C}_M$ classes under derivation will always be implied). If $f$ and the $a_1,\ldots,a_d$ are of class $\mathcal{C}_M$, is it possible to achieve Weierstrass division with $q$ and  $r_0,\ldots,r_{d-1}$ in the same class, or maybe in some larger class $\mathcal{C}_N$ depending on $M$ and on the polynomial $P$? In the absence of extra assumptions, it quickly turns out that, in general, one cannot expect that the quotient and the remainder will be in the same class as $f$, even if the coefficients $a_1,\ldots,a_d$ are real-analytic: see, for instance, Proposition 2 in Section 3 of \cite{Th2}. 

When the class $\mathcal{C}_M$ is quasianalytic, it was proved by Childress \cite{Chi} that in order to have a division property without loss of regularity, that is, with the quotient and the remainder in the same class $\mathcal{C}_M$ as $f$, it is necessary that $P$ be hyperbolic (this condition means that for each value of the parameter $t$, all the roots of $P(\cdot,t)$ are real). Interestingly, it turns out that hyperbolicity is also a sufficient condition for division without loss of regularity, as shown later by Chaumat-Chollet \cite{CC2}.

In the present paper, instead of quasianalytic classes, we shall concentrate on the case of so-called \emph{strongly regular} (in particular, non-quasianalytic) classes, a typical example of which is provided by Gevrey classes $\mathcal{G}^{1+\alpha}$ associated with $M_j=(j!)^\alpha$ for some real $\alpha>0$. A first important result in this direction is due to Bronshtein \cite{Bro}, who proved that for $\mathcal{G}^{1+\alpha}$ data, the quotient and the remainder can be chosen of class $\mathcal{G}^{1+d\alpha}$. This  was later considerably extended by Chaumat-Chollet \cite{CC1}, who showed in particular that if $f$ and $a_1,\ldots,a_d$ belong to a strongly regular class $\mathcal{C}_M$, the quotient and remainder can be chosen of class $\mathcal{C}_N$  with $ N_j= M_{dj}$, or equivalently, $N_j= (M_j)^d$. This result appears in \cite{CC1} as a corollary of a Weierstrass division theorem for generic polynomials $ \Pi(x,\lambda)=x^d+\lambda_1 x^{d-1}+\cdots+\lambda_d $, using the fact that $ P(x,t)=\Pi(x, a_1(t),\ldots,a_d(t))$. Since the exponent $d$ in $(M_j)^d$ is actually the best possible for the generic polynomial of degree $d$, this approach is unlikely to provide better estimates taking into account particular geometric or algebraic features of a given Weierstrass polynomial $P$. 

In the case of a hyperbolic polynomial $P$, using a combination of more direct proofs of Weierstrass division and of specific information on the regularity of the roots of $P$, it is in fact possible to obtain $N_j=M_j$: this was proved by Bronshtein \cite{Bro} for Gevrey classes, and by Chaumat-Chollet in the general case, as the article \cite{CC2} actually encompasses both quasianalytic and non-quasianalytic situations. 

This suggests that it should be also possible to improve the general estimate $N_j=(M_j)^d$ for suitable non-hyperbolic polynomials $P$, in terms of algebraic or geometric features of $P$. In this article, we shall indeed obtain such an improvement, provided $P$ is real-analytic and the locus $\Gamma$ of its complex roots has a sufficiently well-behaved geometry, in a sense described precisely in Section \ref{geometry}. Although this is only a partial answer, it covers a number of standard examples. We shall get $ N_j=(M_j)^\sigma$ where $\sigma$ is a rational number which depends on $P$ and satisfies $1\leq \sigma \leq d$. This number 
is introduced as a {\L}ojasiewicz exponent associated with metric properties of $\Gamma$ and of the sets $\mathcal{N}_z$ of complex parameters $\tau$ such that $P(z,\tau)=0$. For instance, for $m=1$ and $P(x,t)=x^d-t^2$, we have $\sigma=d/2$ and our result therefore shows that if the function germ $f$ is of Gevrey class $\mathcal{G}^{1+\alpha}$, the quotient and the remainder for the Weierstrass division of $f$ by $P$ can be chosen with $\mathcal{G}^{1+\frac{d}{2}\alpha}$ regularity, instead of the cruder $\mathcal{G}^{1+d\alpha}$ estimate given by previously known results. We shall also see that the improved estimates are actually optimal: for certain germs $f$, the quotient and remainder may not be chosen in any class $\mathcal{C}_N$ smaller than $\mathcal{C}_{M^\sigma}$. 

The paper is organized as follows. In Section \ref{DCclasses}, we recall the necessary definitions and fundamental facts pertaining to Denjoy-Carleman classes. In Section \ref{geometry}, we gather all the geometric material that will be used, introducing in particular the assumptions on $P$, and the aforementioned {\L}ojasiewicz exponent $\sigma$. Several examples are provided to illustrate the assumptions and explicit computations of $\sigma$. Section \ref{dbarflatext} is devoted to results on $\bar\partial$-flat extensions of functions: these results are a technical tool in our proof of the division formula with improved estimates, which is eventually carried out, and discussed, in Section \ref{divest}.  
 
\section{Denjoy-Carleman classes}\label{DCclasses}

\subsection{Notation} For any multi-index $J=(j_1,\ldots,j_n)$ of $\mathbb{N}^n$, we always denote the length $j_1+\cdots+j_n$ of $J$ by the corresponding lower case letter $j$. We put $J!=j_1!\cdots j_n!$, $D_x^J=\partial^j/\partial x_1^{j_1}\cdots\partial x_n^{j_n}$ and $x^J=x_1^{j_1}\cdots x_n^{j_n}$. We denote by $\vert \cdot\vert$ the euclidean norm on $\mathbb{R}^n$; distances in $\mathbb{R}^n$ are considered with respect to that norm. For any real $r>0$, we put $I_r= [-r,r]$ 
and $D_r=\{z\in\mathbb{C}: \vert z\vert\leq r\}$. The integer part of a real number $s$ will be denoted by $\lfloor s\rfloor$. 

\subsection{Some properties of sequences}\label{sequences}
Let $ M=(M_j)_{j\geq 0} $ be a sequence of real numbers satisfying the following assumptions:
\begin{equation}\label{norm}
\text{ the sequence } M \text{ is increasing, with } M_0=1,
\end{equation}
\begin{equation}\label{logc}
\text{ the sequence } M \text{ is \emph{logarithmically convex}}.
\end{equation}
Property \eqref{logc} amounts to saying that $ M_{j+1}/M_j $ is increasing. Together with \eqref{norm}, it implies
\begin{equation*}
M_jM_k\leq M_{j+k}\ \textrm{ for any } (j,k)\in\mathbb{N}^2. 
\end{equation*}
We say that the \emph{moderate growth} property holds if there is a constant $ A>0$ such that, conversely, 
\begin{equation}\label{modg}
M_{j+k}\leq A^{j+k} M_jM_k\ \textrm{ for any } (j,k)\in\mathbb{N}^2. 
\end{equation}
We say that $M $ satisfies the \emph{strong non-quasianalyticity} condition if there is a constant $A>0 $ such that
\begin{equation}\label{snqa}
\sum_{j\geq k}\frac{M_j}{(j+1)M_{j+1}}\leq A \frac{M_k}{M_{k+1}} \textrm{ for any } k\in\mathbb{N}.
\end{equation}
Property \eqref{snqa} is indeed stronger than the classical Denjoy-Carl\-eman non-quas\-ian\-alyt\-icity condition
\begin{equation*}
\sum_{j\geq 0}\frac{M_j}{(j+1)M_{j+1}}<\infty.
\end{equation*}
The sequence $M$ is said to be \emph{strongly regular} if it satisfies \eqref{norm}, \eqref{logc}, \eqref{modg} and \eqref{snqa}.  

\begin{exam}\label{exgev}
Let $ \alpha $ and $\beta$ be real numbers, with $ \alpha> 0 $. The sequence $M$ defined by $ M_j=(j!)^\alpha(\ln(j+e))^{\beta j} $ is strongly regular. This is the case, in particular, for Gevrey sequences $ M_j=(j!)^\alpha $.
\end{exam}

With every sequence $ M $ satisfying \eqref{norm} and \eqref{logc} we also associate the function $ h_M $ defined by $ h_M(t)=\inf_{j\geq 0}t^jM_j $ for any real $ t>0 $, and $ h_M(0)=0 $. From \eqref{norm} and \eqref{logc}, it is easy to see that the function $h_M$ is continuous, increasing, and it satisfies $ h_M(t)>0$ for $t>0$ and $ h_M(t)=1 $ for $t\geq 1/M_1 $. It also fully determines the sequence $M$, since we have 
\begin{equation}\label{Legendre}
 M_j=\sup_{t>0}t^{-j}h_M(t). 
\end{equation}

\begin{exam} Let $M$ be as in Example \ref{exgev}, and set $\eta(t)=\exp(-(t\vert\ln t\vert^\beta)^{-1/\alpha})$ for $t>0$. Elementary computations show that there are constants $a>0$, $b>0$ such that $\eta(at)\leq h_M(t)\leq \eta(bt)$ for $t\to 0$. 
\end{exam} 

The moderate growth assumption \eqref{modg} implies that for any real $s\geq 1$, there is a constant $ \kappa_s\geq 1 $, depending only on $ M $, such that
\begin{equation}\label{hfunct2}
h_M(t)\leq \big(h_M(\kappa_s t)\big)^s\text{ for any }t\geq 0.
\end{equation}
We refer to \cite{CC} for a proof. Using \eqref{modg}, \eqref{Legendre} and \eqref{hfunct2}, it is not difficult to see that there are constants $A_1>0$ and $A_2>0$, depending only on $M$ and $s$, such that
\begin{equation}\label{equivseq}
A_1^{j+1} (M_j)^s\leq M_{\lfloor sj\rfloor}\leq A_2^{j+1}(M_j)^s\text{ for any }j\in\mathbb{N}.
\end{equation}
Besides, if $M$ is strongly regular, the sequence $M^s=((M_j)^s)_{j\geq 0}$ is also strongly regular.  

\subsection{Denjoy-Carleman classes}
Let $\Omega$ be an open subset of $\mathbb{R}^n $, and let $M$ be a sequence of real numbers satisfying \eqref{norm} and \eqref{logc}. We say that a function $f$ is of class $ \mathcal{C}_M$ in $\Omega$ if $ f$ belongs to $ \mathcal{C}^\infty(\Omega) $ and, for any compact subset $ K $ of $ \Omega $, one can find a constant $C\geq 0 $, such that the estimate
\begin{equation}
\vert D_x^Jf(x)\vert \leq C^{j+1} j!M_j
\end{equation}
holds for any multi-index $ J\in \mathbb{N}^n$ and any $x\in K $. 

The \emph{Denjoy-Carleman class} associated with $M$ in $\Omega$ is the space, denoted by $\mathcal{C}_M(\Omega)$, of all functions of class $\mathcal{C}_M$ in $\Omega$. A function germ at the origin in $\mathbb{R}^n$ is said to be of class $\mathcal{C}_M$ if it has a representative in $\mathcal{C}_M(\Omega)$ for some neighborhood $\Omega$ of $0$. 
Conditions \eqref{norm} and \eqref{logc} imply that $\mathcal{C}_M(\Omega)$ is an algebra, and that the set of germs of class $\mathcal{C}_M$ at the origin is a local ring. They also imply that $\mathcal{C}_M$ regularity is stable under composition with $\mathcal{C}_M$ maps, see for instance \cite{Rou}. The moderate growth property \eqref{modg} implies the weaker condition
\begin{equation*}
M_{j+1}\leq A^{j+1}M_j\ \textrm{ for any } j\in\mathbb{N}, 
\end{equation*}
which characterizes the stability of $\mathcal{C}_M(\Omega)$ under derivation. For a more detailed overview of the relationship between conditions on $M$ and properties of the corresponding Denjoy-Carleman class, we refer the reader to \cite{Th2}. 

From now on, we will assume that the sequence $M$ is strongly regular. In particular, $\mathcal{C}_M$ regularity is stable under derivation as well as composition with analytic maps, and non-quasianalyticity implies the existence of truncation functions in $\mathcal{C}_M(\Omega)$.  

\section{Roots of polynomials and geometric properties}\label{geometry}
\subsection{Geometric objects associated with a Weierstrass polynomial}\label{objects}
We consider a polynomial 
\begin{equation*}
P(x,t)=x^d+a_1(t)x^{d-1}+\cdots+a_d(t)
\end{equation*} 
where the coefficients $ a_1,\ldots,a_d$ are germs of real analytic functions at the origin in $\mathbb{R}^m $, and $ a_j(0)=0$ for $ j=1,\ldots, d$. We assume that at least one of the coefficients is non identically zero (in the trivial case $ P(x,t)=x^d $, Weierstrass division boils down to a direct application of the Taylor formula with integral remainder). 

By natural complexification, each $a_j$ can be regarded as an holomorphic function germ in a neighborhood of $0$ in $\mathbb{C}^m$. We then have the following lemma, which is a variation on basic results about Weierstrass polynomials (see e.g. \cite{Ran}, Lemma I.3.2). 

\begin{lem}\label{Rouche} 
For each sufficiently small $\eta>0$, one can find $ \delta>0 $ such that for any $z\in D_\delta$, there exists $ \tau \in D_\eta^m$ for which $ P(z,\tau)=0$.  
\end{lem}
\begin{proof} Let $k$ be the greatest integer in $\{1,\ldots,d\}$ such that $a_k$ is not identically zero. For $z\in \mathbb{C}$ and $\tau $ in a neighborhood of $0$ in $\mathbb{C}^m$, we have $ P(z,\tau)=z^{d-k}Q(z,\tau)$ with $ Q(z,\tau)=z^k+a_1(\tau)z^{k-1}+\cdots+a_k(\tau) $. Since $ Q(0,\tau)=a_k(\tau)$ with $a_k\not\equiv 0$, there is an integer $p\geq 1$ and a linear change of variables $\mathcal{L} :\mathbb{C}^m\to \mathbb{C}^m$ such that 
$Q(0,\mathcal{L}(0,\ldots,0,w_m))= c(w_m)w_m^p $ with $ c(0)\neq 0$. Given a sufficiently small $\varepsilon>0$,  Rouch\'e's theorem then yields $\delta>0$ such that for $ z\in D_\delta $ and $(w_1,\ldots,w_{m-1})\in D_\delta^{m-1}$, the function $ w_m\mapsto Q(z,\mathcal{L}(w)) $ has $p$ roots in $D_\varepsilon$. We can obviously assume $\delta\leq\varepsilon$, and choose $\varepsilon $ such that $ \mathcal{L}(D_\delta^{m-1}\times D_\varepsilon)\subset D_\eta^m$. Therefore, we see that, for $\vert z\vert \leq \delta$, the function $\tau\mapsto P(z,\tau)$ has zeros in $D_\eta^m$, hence the result. 
\end{proof}

In what follows, $\eta$ and $\delta$ are chosen as in the statement of Lemma \ref{Rouche} above. For any $z\in D_\delta$, we define
\begin{equation*}
\mathcal{N}_z=\{\tau\in D_\eta^m : P(z,\tau)=0\}. 
\end{equation*}
Roughly speaking, $\mathcal{N}_z$ is the set of all sufficiently small complex parameters $\tau$ for which $z$ is a root of $ P(\cdot,\tau)$. Lemma \ref{Rouche} ensures that $\mathcal{N}_z$ is a non-empty closed subset of $\mathbb{C}^m$. We can therefore define
\begin{equation*}
\rho(z)=d(\mathcal{N}_z,\mathbb{R}^m).
\end{equation*}
Notice that the inclusion $\mathcal{N}_z\subset D_\eta^m$ easily implies
\begin{equation}\label{restrict}
 \rho(z)=d(\mathcal{N}_z,I_\eta^m). 
\end{equation}
We also consider the compact subset of $\mathbb{C}$ defined by  
\begin{equation*}
\Gamma= \{z\in \mathbb{C} : \exists t\in I_\eta^m : P(z,t)=0\}.
\end{equation*}
In simple terms, $\Gamma$ is the locus of all complex roots of $P(\cdot,t)$ when the parameter $t$ varies in a suitable neighborhood of $0$ in $\mathbb{R}^m$.   
 
\subsection{A metric property of the roots}
We start with preparatory lemmas about subanalyticity properties of distance functions. They require only basic properties of subanalytic geometry, for which we refer the reader to \cite{DS}. 
  
\begin{lem}\label{sub1}
The function $ z\mapsto d(z,\Gamma) $ is subanalytic in $\mathbb{C}$. 
\end{lem}
\begin{proof}
The set $Z=\{(z,t)\in \mathbb{C}\times I_\eta^m : P(z,t)=0\} $ is a bounded semianalytic subset of $\mathbb{C}\times\mathbb{R}^m $. The set $\Gamma$ is therefore subanalytic in $\mathbb{C}$ as the projection of $Z$ on the first factor, and the lemma follows (see e.g. Proposition 3.5 of \cite{DS}). 
\end{proof} 

\begin{lem}\label{sub2}
The function $ \rho $ is subanalytic in $ D_\delta$.
\end{lem}
\begin{proof} The argument is a slight variation on the proof of Proposition 3.5 of \cite{DS}. Remark that the inclusion $\mathcal{N}_z\subset D_\eta^m$ implies $d(\mathcal{N}_z,\mathbb{R}^m)\leq \eta'$ with $\eta'=\sqrt{m}\eta$. The graph of the function thus 
can be described as $X\setminus Y$, where $X= \{(z,s)\in D_\delta\times I_{\eta'} : d(\mathcal{N}_z,\mathbb{R}^m)\leq s\}$ and $Y=\{(z,s)\in D_\delta\times I_{\eta'} : d(\mathcal{N}_z,\mathbb{R}^m)< s\}$. For $z\in D_\delta$, taking into account \eqref{restrict} and the closedness of $\mathcal{N}_z$, we see that the condition $d(\mathcal{N}_z,\mathbb{R}^m)\leq s$ holds if and only if one can find $t\in I_\eta^m$ and $\tau\in D_\eta^m$ such that $P(z,\tau)=0$ and $\vert t-\tau \vert\leq s$. Put  
$\widetilde{X}= \{(z,s,t,\tau)\in D_\delta\times I_{\eta'}\times I_\eta^m\times D_\eta^m : P(z,\tau)=0  \textrm{ and }\vert t-\tau\vert\leq s\}$. The set $\widetilde{X}$ is a bounded semianalytic subset of $\mathbb{C}\times\mathbb{R}\times\mathbb{R}^m\times\mathbb{C}^m$ and $X$ is therefore subanalytic since it is a projection of $\widetilde{X}$. Analogously, $Y$ is subanalytic, as well as $X\setminus Y$, which proves the result. 
\end{proof}

We are now able to introduce the {\L}ojasiewicz exponent $\sigma$ that will play a crucial role in our division estimates. Let $ \mathcal{S} $ be the set of real numbers $s\geq 1$ for which there exists a constant $c>0$ such that 
\begin{equation}\label{ineqloj}
\rho(z)\geq c\,d(z,\Gamma)^s\quad \textrm{for any } z \in D_\delta. 
\end{equation}

\begin{prop} The set $\mathcal{S}$ is non-empty and the number $ \sigma= \inf\mathcal{S} $ belongs to $\mathcal{S}$, in other words there is a constant $c>0$ such that property \eqref{ineqloj} holds with $s=\sigma $. Moreover, the exponent $\sigma$ is rational and satisfies $ 1\leq\sigma\leq d $.
\end{prop}
 \begin{proof}  
 Both functions $ z\mapsto d(z,\Gamma)$ and $\rho$ are subanalytic on $D_\delta$ by Lemmas \ref{sub1} and \ref{sub2}. 
Moreover, by \eqref{restrict}, $\rho(z)=0$ implies the existence of $t\in I_\eta^m$ such that $ P(z,t)=0$, hence $ z\in \Gamma$, and $ d(z,\Gamma)=0$. The first statement of the proposition, and the fact that $\sigma$ is rational, are therefore consequences of classical results on the {\L}ojasiewicz inequality for subanalytic functions \cite{BR}. We now proceed to prove the inequality $1\leq \sigma\leq d$. For $t\in I_\eta^m$, denote by $ \mu_1(t),\ldots,\mu_d(t)$ the roots of $ P(\cdot,t)$, counted with multiplicities, so that $ P(z,t)=\prod_{j=1}^d(z-\mu_j(t))$. Each $\mu_j(t)$ belongs to $\Gamma$, hence $ \vert z-\mu_j(t)\vert\geq d(z,\Gamma) $ and 
\begin{equation}\label{minP}
\vert P(z,t)\vert\geq d(z,\Gamma)^d.
\end{equation}
By the mean value theorem, there is also a constant $C>0$ such that $ \vert P(z,t)-P(z,\tau)\vert\leq C \vert t-\tau\vert $ for any $z\in D_\delta$, $t\in I_\eta^m$ and $\tau\in D_\eta^m$. Choosing $\tau \in \mathcal{N}_z$ such that $ \vert t-\tau\vert=d(t,\mathcal{N}_z)$, we get $ \vert P(z,t)\vert\leq C d(t,\mathcal{N}_z)$. Thanks to \eqref{minP}, we derive $ d(z,\Gamma)^d\leq Cd(t,\mathcal{N}_z)$. Taking the infimum with respect to $t\in I_\eta^m$ then yields $ d(z,\Gamma)^d\leq C\rho(z)$. Thus, $d$ belongs to $\mathcal{S}$, hence $ \sigma\leq d$. We also have $\sigma\geq 1$ from the very definition of $\mathcal{S}$. 
\end{proof}

Examples of explicit computations of the exponent $\sigma$ will be given in Section \ref{examp}. Prior to this, we state a technical proposition that will be a key ingredient in the proof of the main result.  

\begin{prop}\label{Cauchy} 
Let $V$ be a bounded open neighborhood of $ D_\delta\cup \Gamma$ in $\mathbb{C}$. 
With the same notations as above, there are real numbers $ \varepsilon\in ]0,\eta[ $, $\nu\geq 1$ and $C\geq 0$ such that, for any multi-index $L\in \mathbb{N}^n$ of length $l$ and any $(z,t)\in (V\setminus \Gamma)\times  I_{\varepsilon}^m$, we have
\begin{equation}\label{estiminvP}
\left\vert D_t^L \left(\frac{1}{P(z,t)}\right)\right\vert\leq C^{l+1}l!\, d(z,\Gamma)^{-(\sigma l+\nu)}.
\end{equation}
\end{prop} 
\begin{proof} We first prove \eqref{estiminvP} for $(z,t)\in (D_\delta\setminus \Gamma)\times I_\varepsilon^m$. Proceeding as in the proofs of Lemmas \ref{sub1} and \ref{sub2}, it is not difficult to see that the functions $ (z,\tau)\mapsto \vert P(z,\tau)\vert $ and $(z,\tau)\mapsto d(\tau,\mathcal{N}_z)$ are subanalytic in $ D_\delta\times D_\eta^m$, and that the condition $P(z,\tau)=0$ implies $ d(\tau, \mathcal{N}_z)=0$. Thus, for some suitable constants $ C_1>0$ and $\alpha>0$, we have a {\L}ojasiewicz inequality 
\begin{equation}\label{loja1}
\vert P(z,\tau)\vert\geq C_1 d(\tau,\mathcal{N}_z)^\alpha\quad \textrm{for } (z,\tau)\in D_\delta\times D_\eta^m.
\end{equation}
 Set $a=\frac{1}{2m}$. For $t\in I_{a\eta}^m$, we consider the polydisc $P_t = \{\tau \in \mathbb{C}^m: \vert \tau_j-t_j\vert\leq a\rho(z),\ j=1,\ldots,m\}$. 
For any $\tau \in P_t$, we have $\vert t-\tau\vert\leq \frac{1}{2\sqrt{m}}\rho(z)$ and $ \vert\tau\vert \leq \vert t\vert + \frac{1}{2\sqrt{m}}\rho(z)\leq a\sqrt{m}\eta+\frac{1}{2}\eta\leq \eta $, which implies $\tau\in D_\eta^m $. We also have $d(\tau,\mathcal{N}_z)\geq d(t,\mathcal{N}_z)-\vert t-\tau\vert\geq \rho(z)-\vert t-\tau\vert\geq \frac{1}{2}\rho(z) $. Together with \eqref{loja1}, these considerations imply  
\begin{equation*}
\vert P(z,\tau)\vert\geq C_2 \rho(z)^\alpha \quad\textrm{for any } (z,t)\in D_\delta\times I_{a\eta}^m\textrm{ and }\tau\in P_t,
\end{equation*}
with $C_2= 2^{-\alpha}C_1$. For $z\in D_\delta\setminus \Gamma$, property \eqref{ineqloj} implies $\rho(z)>0$ and the Cauchy formula applied on $P_t$ then yields 
\begin{equation*}
\left\vert D_t^L \left(\frac{1}{P(z,t)}\right)\right\vert\leq C_3^{l+1}l!\,\rho(z)^{-(l+\alpha)}
\end{equation*}
for some suitable $C_3>0$. In view of \eqref{ineqloj}, this proves the desired estimate with $ \nu=\sigma\alpha$ and some suitable $C$, provided we assume $\varepsilon \leq a\eta $. 
  
We now claim that inequality \eqref{estiminvP} is still valid for $(z,t)\in ((V\setminus D_\delta)\setminus\Gamma)\times I_\varepsilon^m$, up to a modification of $C$ and $\varepsilon$. Indeed, given a real $\delta'$ with $ 0<\delta'<\delta$, a standard argument based on Rouch\'e's theorem  yields $\eta'>0 $ such that for any $t\in I_{\eta'}^m $, all the roots of $P(\cdot,t)$ belong to $ D_{\delta'}$. If we assume $\varepsilon<\eta'$, it follows that $ 1/P$ is real-analytic in a neighborhood of the compact subset $ \overline{(V\setminus D_\delta)}\times I_{\varepsilon}^m$ in $\mathbb{C}\times \mathbb{R}^m$. For $(z,t)\in (V\setminus D_\delta)\times I_\varepsilon^m$, we therefore have the standard Cauchy estimate
\begin{equation*}
\left\vert D_t^L \left(\frac{1}{P(z,t)}\right)\right\vert\leq C_4^{l+1}l! 
\end{equation*}
for some suitable $C_4>0$. Since $d(\cdot,\Gamma)$ is bounded in $V$, the claim follows, and the proof is complete.  
\end{proof}

We shall now make additional geometric assumptions on $\Gamma$. These assumptions will be used to obtain certain $\bar\partial$-flat extensions in Section \ref{dbarflatext}.

\subsection{Geometric assumptions on $\Gamma$}\label{geomgamma}
Given a real number $\mu\geq 1$, we recall that two compact subsets $E_1$ and $E_2$ of $\mathbb{R}^n$ with non-empty intersection are said to be \emph{$\mu$-regularly separated} if there is a constant $C>0$ such that 
\begin{equation*}
d(x,E_1)\geq C d(x,E_1\cap E_2)^\mu \textrm{ for any } x\in E_2.
\end{equation*} 
The condition is easily seen to be symmetrical with respect to $E_1$ and $E_2$. We refer the reader to \cite{DS, Tou} for further information. In what follows, we shall be dealing with the case $\mu=1$. 

We say that a compact subset $S$ of $\mathbb{C}$ is an \emph{admissible arc} if there is a biholomorphic map $u$ in an open neighborhood of the closed unit disk in $\mathbb{C}$ such that $u(0)=0$ and $S=u([0,1])$. For example, it is easy to see that any germ of regular real-analytic arc with endpoint $0$ in $\mathbb{C}$ is represented by an admissible arc in a sufficiently small neighborhood of $0$. 

From now on, it will always be assumed that we have
\begin{equation*}
\Gamma \subset\bigcup_{j=0}^N \Gamma_j,
\end{equation*}
where the sets $\Gamma_0,\ldots,\Gamma_N$ enjoy the following properties: 
\begin{enumerate}[(i)]
\item $\Gamma_0=I_\delta\cup (\Gamma\cap\mathbb{R})$,
\item For $1\leq j\leq N$, the set $\Gamma_j$ is an admissible arc,
 \item For $0\leq j\leq N$, $0\leq k\leq N$ and $j\neq k$, we have $\Gamma_j\cap\Gamma_k=\{0\}$ and the sets $\Gamma_j$ and $\Gamma_k$ are $1$-regularly separated. 
\end{enumerate}
Examples are studied in Section \ref{examp} below.
 
\subsection{Explicit examples}\label{examp}
 
We start with a family of polynomials for which we have $\sigma=1$. 
 
\begin{exam}\label{examp1} Set $m=1$ and $ P(x,t)=x^2+t^{2p} $ for some integer $p\geq 1$. The set $\Gamma $ is the segment $ [-i\lambda_0, i\lambda_0]$ on the imaginary axis, with $\lambda_0=\eta^p$. The assumptions \ref{geomgamma} are clearly verified with $N=2$, $\Gamma_0=I_\delta$, $\Gamma_1=[-i\lambda_0,0]$ and $\Gamma_2=[0,i\lambda_0]$.

We now proceed to compute $\sigma$. For $z\in D_\delta$, put $z=re^{i\theta}$ with $ \theta\in \mathbb{R}$ and $r=\vert z\vert$. Setting $ \varphi=\frac{\theta}{p}+\frac{\pi}{2p}$, we have $\mathcal{N}_z=\{\pm r^{1/p}e^{i(\varphi+k\frac{\pi}{p})} : k=0,\ldots, p-1\}$, which implies $ \rho(z)= r^{1/p}\min_{0\leq k\leq p-1} \vert\sin(\varphi+k\frac{\pi}{p})\vert  $. 
We also have $ d(z,\Gamma)=r\vert \cos\theta\vert= r\vert\sin (p\varphi)\vert=2^{p-1}r\prod_{k=0}^{p-1}\vert \sin(\varphi+k\frac{\pi}{p})\vert \leq 2^{p-1}r \min_{0\leq k\leq p-1} \vert\sin(\varphi+k\frac{\pi}{p})\vert $. Hence, we get $ \rho(z)\geq (2r^{1/p})^{1-p} d(z,\Gamma)\geq (2\delta^{1/p})^{1-p} d(z,\Gamma) $. Thus, \eqref{ineqloj} holds for $s=1$, which implies $\sigma=1$. 
 \end{exam}

The next example provides a family of polynomials for which $\sigma$ equals half the degree.

\begin{exam}
Assume $m=1$ and $ P(x,t)=x^d-t^2 $ for some integer $d\geq 2$.  
For $\varphi\in \mathbb{R}$, denote by $\Delta_\varphi$ the segment $ \{ \lambda e^{i\varphi}: 0\leq\lambda\leq\lambda_0\}$ with $\lambda_0=\eta^{2/d}$. We then have $\Gamma =\bigcup_{k=0}^{d-1} \Delta_{\varphi_k} $ with $ \varphi_k=\frac{2k\pi}{d}$. It is therefore easy to see that $\Gamma$ satisfies the geometric assumptions \ref{geomgamma}. 

For $z\in D_\delta$, put $ z=re^{i\theta}$ with $\theta\in [0,2\pi[$ and $r=\vert z\vert$. We then have $\mathcal{N}_z=\{\pm r^{d/2}e^{id\theta/2}\} $ and $\rho(z)= r^{d/2}\vert\sin(\frac{d\theta}{2})\vert $.
Observe that, for any $k\in \{0,\ldots, d-1\} $, we have $ \vert\sin(\frac{d\theta}{2})\vert =\vert\sin(\frac{d}{2}(\theta-\varphi_k)\vert $. Moreover, it is possible to choose $k$ such that $ \vert \theta-\varphi_k\vert\leq \frac{\pi}{d} $. Hence, we get $ \vert\frac{d}{2}(\theta-\varphi_k))\vert\leq \frac{\pi}{2} $ and $ \vert\sin(\frac{d}{2}(\theta-\varphi_k))\vert\geq \frac{d}{\pi} \vert \theta-\varphi_k\vert$ thanks to the elementary inequality $ \vert \sin u\vert \geq \frac{2}{\pi}\vert u\vert $ for $\vert u\vert\leq \frac{\pi}{2}$. Thus, we obtain $ \rho(z)\geq \frac{d}{\pi} r^{d/2}\vert \theta-\varphi_k\vert $. We also have $ d(z,\Gamma)\leq d(z,\Delta_{\varphi_k})= r\vert \sin(\theta-\varphi_k)\vert\leq r \vert \theta-\varphi_k\vert$. Gathering the inequalities, we obtain $ \rho(z)\geq \frac{d}{\pi} d(z,\Gamma)^{d/2} \vert \theta-\varphi_k\vert^{1-\frac{d}{2}}\geq c d(z,\Gamma)^{d/2} $ with $c=(d/\pi)^{d/2}$. Thus, condition \eqref{ineqloj} holds for $s=d/2$. Moreover, for $z=r e^{\frac{i\pi}{d}}$ we have $ d(z,\Gamma)=r\sin(\frac{\pi}{d}) $ and $\rho(z)= r^{d/2} $. Letting $r\to 0$, we see that \eqref{ineqloj} implies $ s\geq d/2$. Therefore, we have $\sigma=d/2$. 
\end{exam}

It is natural to seek a connection between $\sigma$ and the properties of hyperbolic polynomials used in  \cite{Bro} or \cite{CC2}. The case $d=2$ in the preceding example corresponds to the very simple hyperbolic polynomial $x^2-t^2$, which has smooth roots, and for which we have obtained $\sigma=1$. We now study a less trivial hyperbolic example in higher dimensions of the parameter $t$, with non-smooth roots. Note that, in the case of hyperbolic polynomials, the geometric assumptions \ref{geomgamma} are always trivially verified (with $N=0$) since we have $ \Gamma\subset \mathbb{R}$. 
 
\begin{exam} Set $m=2$ and $P(x,t)=x^2-(t_1^2+t_2^2)$. We then clearly have $\Gamma=[-\lambda_0,\lambda_0]$ with $\lambda_0=\sqrt{2}\eta$. For $z\in D_\delta$ with $\delta>0$ small enough, we have $ d(z,\Gamma)=\vert\Im z\vert$.  

For $\tau=(\tau_1,\tau_2)\in\mathbb{C}^2$, we set $t_j=\Re \tau_j$ and $s_j=\Im \tau_j$ for $j=1,2$. Given $ \alpha\in \mathbb{R}$, set $ L_\alpha=\{\tau \in \mathbb{C}^2 : s_2=\alpha s_1\}$. We shall now prove that we have
\begin{equation}\label{intermed}
d(\tau,\mathbb{R}^2)\geq d(z,\Gamma)\textrm{ for any } \tau \in \mathcal{N}_z\cap L_\alpha.
\end{equation}
  
For $\tau\in L_\alpha$, we have $d(\tau,\mathbb{R}^2)^2= (1+\alpha^2)s_1^2$, and the point $\tau$ belongs to $\mathcal{N}_z$ if and only if it satisfies the equations 
\begin{equation}\label{extrem}
\left \{
\begin{array}{r c l}
    t_1^2+t_2^2-(1+\alpha^2)s_1^2 & = & \Re (z^2) \\
    2(t_1+\alpha t_2)s_1& = & \Im(z^2).\\
\end{array}
\right.
\end{equation}
Thus, for a given $z$, we are looking for a minorant of $f(s_1)=(1+\alpha^2)s_1^2$ on the subset $C_\alpha$ of all points $(t_1,t_2,s_1)\in \mathbb{R}^3$ satisfying \eqref{extrem}. We have to discuss two cases. 

\textit{First case:} $\Im(z^2)=0$. Either $z$ is real and the estimate $f(s_1)\geq d(z,\Gamma)^2$ is trivial since $d(z,\Gamma)=0$, or we have $z=iy$ with $y\neq 0$. Then, we get $\Re(z^2)<0$ and the first equation in \eqref{extrem} implies $ s_1\neq 0$. The second equation in \eqref{extrem} then yields $t_1=-\alpha t_2$. Thus, the first equation becomes $ (1+\alpha^2)(s_1^2-t_2^2)= y^2$, and we derive $ f(s_1)\geq y^2=d(z,\Gamma)^2$. 

\textit{Second case:} $\Im(z^2)\neq 0$. In this case, the equations \eqref{extrem} imply that $C_\alpha$ is a submanifold of $\mathbb{R}^3$ and a classic argument of Lagrange multiplier shows that if $f$ has a minimum at $(t_1,t_2,s_1)\in C_\alpha$, there are real numbers $\lambda $ and $\mu$ such that 
\begin{equation}\label{extrem1}
\left \{
\begin{array}{r c l}
    t_1\lambda+s_1\mu & = & 0 \\
    t_2\lambda+\alpha s_1\mu & = & 0\\
    -(1+\alpha)^2s_1\lambda+ (t_1+\alpha t_2)\mu & = & (1+\alpha^2)s_1.
\end{array}
\right.
\end{equation}
Note that $(\lambda,\mu)\neq (0,0)$, otherwise \eqref{extrem1} would imply $s_1=0$, in contradiction with the second equation of \eqref{extrem}. Thus, the two first equations of \eqref{extrem1} imply $ \det\left(\begin{smallmatrix} t_1 & s_1\\ t_2 & \alpha s_1\\ 
\end{smallmatrix}\right)=0 $, hence $ t_2=\alpha t_1$. By \eqref{extrem}, we now have $ (1+\alpha)^2 \tau_1^2=z^2 $, hence $ \tau_1= \pm (1+\alpha^2)^{-1/2} z $. Thus, we get $s_1=\pm (1+\alpha^2)^{-1/2} \Im z$ and $f(s_1)= (\Im z)^2=d(z,\Gamma)^2$. 

The proof of \eqref{intermed} is now complete, and we can conclude: since the roles of $s_1$ and $s_2$ are symmetric and $\alpha$ is arbitrary, \eqref{intermed} implies $d(\tau,\mathbb{R}^2)\geq d(z,\Gamma)$ for any $\tau \in \mathcal{N}_z$. Hence, we have $\rho(z)\geq d(z,\Gamma)$ and $\sigma=1$. 
\end{exam}

The preceding example suggests a natural question.  

\begin{prob}
It would be interesting to know whether the hyperbolicity of the polynomial $P$ always implies $\sigma=1$. In the light of Theorem \ref{main}, a  positive answer would indeed provide a geometric interpretation of the main result in \cite{CC2}, at least for real-analytic polynomials.
\end{prob}

We conclude this section with two examples of possible obstructions to the geometric assumptions \ref{geomgamma}.

\begin{exam} 
Set $m=2$ and $P(x,t)=x^2-2t_1x+ t_1^2+t_2^2$. For $t\in\mathbb{R}^2$, the complex roots of $P(\cdot,t)$ are $t_1+it_2$ and its conjugate. Therefore, $\Gamma$ is a neighborhood of $0$ in $\mathbb{C}$ and part (ii) of the assumptions \ref{geomgamma} fails. 
\end{exam}

\begin{exam} 
Set $m=1$ and $P(x,t)=x^2-2tx+t^2+t^4$. For $t\in\mathbb{R}$, the complex roots of $P(\cdot,t)$ are $t+it^2$ and its conjugate. Thus, in a neighborhood of $0$, the set $\Gamma$ is the union of two parabolas tangent to the real axis at the origin, and part (iii) of the assumptions \ref{geomgamma} fails. 
\end{exam}

\section{On $\bar\partial$-flat extensions}\label{dbarflatext}

\subsection{Prerequisites on Whitney jets}
Let $F=(F_J)_{J\in\mathbb{N}^n}$ be a family of continuous functions on a compact subset $E$ of $\mathbb{R}^n$. With any point $\xi\in E$ and any integer $l\geq 0$, we associate the Taylor polynomial $ T_\xi^lF(x)=\sum_{j\leq l}\frac{1}{J!}F_J(\xi)(x-\xi)^J$. The family $F$ is called a \emph{Whitney jet of class $\mathcal{C}_M$ on $E$} if there is a positive constant $C$ such that 
\begin{equation*}
\vert F_J(x)\vert\leq C^{j+1}j!M_j
\end{equation*}
for any $x\in E$ and $J\in \mathbb{N}^n$, and 
\begin{equation*}
\vert F_J(x)-D_x^J T_\xi^l F(x)\vert\leq C^{l+1}j!M_{l+1}\vert x-\xi\vert^{l+1-j}
\end{equation*}
for any $(x,\xi)\in E\times E$, any $l\in \mathbb{N}$ and any multi-index $J\in \mathbb{N}^n$ with $j\leq l$. 

Let $f$ be a function of class $\mathcal{C}_M$ in an open neighborhood of $E$. Using the Taylor formula, it is easy to see that $(D_x^Jf)_{J\in\mathbb{N}^n}$ is a Whitney jet of class $\mathcal{C}_M$ on $E$. Conversely, there is a $\mathcal{C}_M$ version of Whitney's extension theorem. 

\begin{prop}\label{WhitExt}\cite{BBMT, B, CC} 
For any strongly regular sequence $M$ and any Whitney jet $F$ of class $\mathcal{C}_M$ on $E$, there is a function $f$ of class $\mathcal{C}_M$ in $\mathbb{R}^n$ such that $F_J(x)= D_x^Jf(x)$ for any $J\in\mathbb{N}^n$ and $x\in E$.
\end{prop}

We shall apply this result to a special case of $\bar\partial$-flat extension of $\mathcal{C}_M$ functions. Let $K$ be a compact subset of $\mathbb{C}$ and let $T$ be a compact subset of $\mathbb{R}^m$. A function $g: (z,t)\mapsto g(z,t)$ of class $C^\infty$ in an open neighborhood of $K\times T$ in $\mathbb{C}\times \mathbb{R}^m$ is said to be \emph{$\bar\partial$-flat with respect to $z$ on $K\times T$} if $\frac{\partial g}{\partial\bar z}$ vanishes, together with all its derivatives, on $K\times T$. In the case $K\subset\mathbb{R}$, we have the following lemma. 

\begin{lem}\label{extflat1} 
Let $K$ be a compact subset of $\mathbb{R}$. 
For any function $f : (x,t)\mapsto f(x,t)$ of class $\mathcal{C}_M$ in an open neighborhood of $K\times T $ in $\mathbb{R}\times\mathbb{R}^m$, there is a function $g : (z,t)\mapsto g(z,t)$ of class $\mathcal{C}_M$ in $\mathbb{C}\times \mathbb{R}^m$, which is $\bar\partial$-flat with respect to $z$ on $K\times T $ and such that
\begin{equation}\label{flat1}
 \frac{\partial^{j}g}{\partial z^j}(x,t)=\frac{\partial^j f}{\partial x^j}(x,t)\textrm{ for any }(x,t)\in K\times T\textrm{ and any } j\in\mathbb{N}.
\end{equation} 
\end{lem}
\begin{proof} Apart from the presence of the parameter $t$, the proof follows a standard pattern for $\bar\partial$-flat extension of Whitney jets. We briefly recall the argument for the reader's convenience. For any $(j,k)\in \mathbb{N}^2$, $Q\in \mathbb{N}^m$ and $(x,t)\in K\times T$, we put $ F_{j,k,Q}(x,t)=i^k D_x^{j+k}D_t^Qf(x,t)$, which amounts to define $F_{j,k,Q}$ via the identity of formal power series  
\begin{equation*}
\sum_{(j,k)\in\mathbb{N}^2 
}F_{j,k,Q}(x,t)\frac{(X-x)^j}{j!}\frac{Y^k}{k!}
=\sum_{l\in\mathbb{N} 
}D_x^lD_t^Qf(x,t)\frac{(Z-x)^l}{l!}
\quad\textrm{with }Z=X+iY.
\end{equation*} 
Using the fact that the family of derivatives $(D_x^lD_t^Qf)_{(l,Q)\in \mathbb{N}\times \mathbb{N}^m}$ is a Whitney jet of class $\mathcal{C}_M$ on $K\times T$ viewed as a compact subset of $\mathbb{R}\times\mathbb{R}^m$, it is not difficult to see that $ F=(F_{j,k,Q})_{(j,k,Q)\in\mathbb{N}^2\times\mathbb{N}^m}$ is a Whitney jet of class $\mathcal{C}_M$ on $K\times T$ viewed as a compact subset of $\mathbb{C}\times\mathbb{R}^m$. It then suffices to apply Proposition \ref{WhitExt} to obtain a function $g$ of class $\mathcal{C}_M$ in $\mathbb{C}\times\mathbb{R}^m$ for which \eqref{flat1} and the $\bar\partial$-flatness with respect to $z$ will be immediate consequences of the aforementioned formal identity. 
\end{proof}

\subsection{An extension property}\label{extensprop}
 The following proposition will be needed in the proof of the division  theorem. The notations are those of Section \ref{objects}, and we assume that $\Gamma$ satisfies the geometric assumptions \ref{geomgamma}. 

\begin{prop}\label{mainextprop}
 Let $f$ be a $\mathcal{C}_M$ function in an open neighborhood of $I_\delta\times I_\eta^m$ in $\mathbb{R}\times \mathbb{R}^m$ and let $V$ be an open neighborhood of $D_\delta\cup \Gamma$ in $\mathbb{C}$. Then there is a function $g$ of class $\mathcal{C}_M$ in $\mathbb{C}\times\mathbb{R}^m$, compactly supported in $V\times  \mathbb{R}^m$, which is $\bar\partial$-flat with respect to $z$ on $ (I_\delta\cup\Gamma)\times I_\eta^m$ and which coincides with $f$ on $I_\delta\times I_\eta^m$.
\end{prop}
\begin{proof} First, we remark that property (ii) in the geometric assumption \ref{geomgamma} implies that there are biholomorphic maps $u_1,\ldots,u_N$ in an open neighborhood $U$ of the closed unit disk in $\mathbb{C}$ such that for any $j=1,\ldots,N$, we have $u_j(0)=0$ and   
\begin{equation}\label{canonical}
\Gamma_j=u_j([0,1]).
\end{equation} 
Now, put $T=I_\eta^m$. We shall prove by induction that for $k=0,\ldots,N$, there is a function $g_k: (z,t)\mapsto g_k(z,t)$ of class $\mathcal{C}_M$ in $\mathbb{C}\times\mathbb{R}^m$ such that 
\begin{equation}\label{indu1}
g_k\textrm{ is }\bar\partial\textrm{-flat with respect to }z\textrm{ on }\bigg(\bigcup_{j=0}^k\Gamma_j\bigg)\times T,
\end{equation}
\begin{equation}\label{indu3}
\frac{\partial^jg_k}{\partial z^j}(x,t)=\frac{\partial^jf}{\partial x^j}(x,t)\textrm{ for any }(x,t)\in I_\delta\times T\textrm{ and any } j\in\mathbb{N}.
\end{equation}
Once this property is established, it suffices, in view of \eqref{indu1}, \eqref{indu3} and of the assumptions \ref{geomgamma}, to set $g(z,t)=\chi(z,t) g_N(z,t) $ where $\chi$ is a truncation function of class $\mathcal{C}_M$, compactly supported in $V\times\mathbb{R}^m$, and such that $\chi(z,t)=1$ for $ (z,t)\in (D_\delta\cup\Gamma)\times T$.  

The existence of $g_0$ is an immediate consequence of Lemma \ref{extflat1} applied with $K=\Gamma_0$, since we can always assume that $f$ is $\mathcal{C}_M$ in a neighborhood of $K\times T$. Assuming that we have obtained $g_k$ for some $k$ with $0\leq k<N$, we proceed with the construction of $g_{k+1}$. 
Put $\Gamma'_k=\bigcup_{j=0}^k\Gamma_j$. The geometric assumptions \ref{geomgamma} imply
\begin{equation}\label{sep1}
 (\Gamma'_k\times T)\cap(\Gamma_{k+1}\times T)=\{0\}\times T
\end{equation}
and
\begin{equation}\label{sep2}
\Gamma'_k\times T \textrm{ and }\Gamma_{k+1}\times T\textrm{ are }1\textrm{-regularly separated}.
\end{equation}
For $(x,t)$ in a neigborhood of $[0,1]\times T$ in $\mathbb{R}\times\mathbb{R}^m$, we consider 
$ h_k(x,t)=g_k(u_{k+1}(x),t) $, where $u_{k+1}$ is the biholomorphic map associated with $\Gamma_{k+1}$ in \eqref{canonical}.  
Taking into account the stability of $\mathcal{C}_M$ regularity under composition with analytic maps, Lemma \ref{extflat1} yields a function $\widetilde{h}_k: (z,t)\mapsto \widetilde{h}_k(z,t) $ of class $\mathcal{C}_M$ in $\mathbb{C}\times\mathbb{R}^m$ which satisfies the following properties:
\begin{equation}\label{h0flat}
 \widetilde{h}_k\textrm{ is }\bar\partial\textrm{-flat with respect to }z\textrm{ on }[0,1]\times T, 
\end{equation}
\begin{equation}\label{h0der}
\frac{\partial^j \widetilde{h}_k}{\partial z^j}(x,t)= \frac{\partial^j h_k}{\partial x^j}(x,t)\textrm{ for any }(x,t)\in [0,1]\times T \textrm{ and } j\in\mathbb{N}.
\end{equation}
Since $\Gamma_{k+1}$ is contained in the image of $u_{k+1}$, we can define $\widetilde{g}_k(z,t)=\widetilde{h}_k(u_{k+1}^{-1}(z),t)$ for any $(z,t)$ in a neighborhood of $\Gamma_{k+1}\times T$ in $\mathbb{C}\times\mathbb{R}^m$. 
Then, by \eqref{h0flat}, the function $\widetilde{g}_k$ is $\bar\partial$-flat with respect to $z$ on $\Gamma_{k+1}\times T$. Also, taking into account \eqref{sep1}, \eqref{h0der}, the induction assumption and the definition of $h_k$ and $\widetilde{g}_k$, it is not difficult to see that we have
\begin{equation}\label{coincide1}
\frac{\partial^j \widetilde{g}_k}{\partial z^j}(0,t)=\frac{\partial^j g_k}{\partial z^j}(0,t)= \frac{\partial^j f}{\partial x^j}(0,t)\textrm{ for any }j\in\mathbb{N}\textrm{ and }t\in T.
\end{equation}
Thus, the functions $g_k$ and $\widetilde{g}_k$ are respectively of class $\mathcal{C}_M$ in neighborhoods of $\Gamma'_k\times T$ and $\Gamma_{k+1}\times T$, and their respective jets on $\Gamma'_k\times T$ and $\Gamma_{k+1}\times T$ coincide on $\{0\}\times 
T$. Taking \eqref{sep1} and \eqref{sep2} into account, Theorem 2.4 of \cite{Th1} then yields a function $g_{k+1}$ of class $\mathcal{C}_M$ in $\mathbb{C}\times\mathbb{R}^m$ which coincides to infinite order with $g_k$ on $\Gamma'_k\times T$ and with $\widetilde{g}_k$ on $\Gamma_{k+1}\times T$. This easily implies that $g_{k+1}$ has the required properties \eqref{indu1} and \eqref{indu3}, which completes the induction. 
\end{proof}

\section{Division with estimates}\label{divest}

\subsection{A division formula} We follow closely, \emph{mutatis mutandis}, the presentation of \cite{CC2}. The resulting division formula will serve as a basis for our estimates. Being given a compact subset $K$ of $\mathbb{C}$, we put $V(K)=\{\zeta\in \mathbb{C}: d(\zeta,K)<1\} $ and 
$\Omega=\{(z,\zeta)\in \mathbb{C}\times (V(K)\setminus K) : d(z,K)<d(\zeta,K)\}$. The following result is then an immediate application of Lemma 3 of \cite{CC1}. 

\begin{lem}\cite{CC1}\label{lemphi} 
There is a constant $ C_0\geq 0$ and a function $\varphi$ of class $\mathcal{C}^\infty$ in $\Omega$ and such that, for any $(z,\zeta)\in \Omega$, we have 
\begin{equation}\label{phi1}
\varphi(z, \zeta)=1\textrm{ if }d(z,K)<\frac{1}{64}d(\zeta,K),
\end{equation} 
\begin{equation}\label{phi2}
\varphi(z, \zeta)=0\textrm{ if }d(z,K)>\frac{3}{4}d(\zeta,K),
\end{equation} 
\begin{equation}\label{phi3}
\left\vert\frac{\partial\varphi}{\partial\bar z}(z,\zeta)\right\vert\leq C_0 d(\zeta,K)^{-1}.
\end{equation}
\end{lem}

In what follows, the notations are those of Section \ref{objects}, and we always assume that $\Gamma$ satisfies the geometric assumptions \ref{geomgamma}. We set
\begin{equation*}
K= I_\delta\cup \Gamma\quad\textrm{and}\quad V=V(K).  
\end{equation*}
Since we can always assume $\delta<1$, we have $D_\delta\cup \Gamma\subset V$. Note, also, that 
 $P(z,t)$ does not vanish in $(V\setminus K)\times I_\eta^m$. 

Elementary considerations show that 
for any $(x,t)\in \mathbb{C}\times I_\eta^m$ and any $(z,t)\in \mathbb{C}\times I_\eta^m$, we have 
\begin{equation}\label{polynom}
P(x,t)-P(z,t)=(x-z)\sum_{j=0}^{d-1}S_j(z,t)x^j,
\end{equation}
where each term $S_j(z,t)$ is a polynomial in $z$ whose coefficients are affine functions of $a_1(t),\ldots,a_d(t)$. In particular, they are real-analytic with respect to $t$ in a neighborhood of $I_\eta^m$.

The following division formula, where $d\lambda$ denotes the plane Lebesgue measure, then comes as an immediate adaptation of Proposition 4 of \cite{CC1} or Lemma 4.2 of \cite{CC2}. 

\begin{prop}\label{divisionCC}\cite{CC1,CC2} 
Let $g$ be a function of class $\mathcal{C}^\infty$ in $\mathbb{C}\times\mathbb{R}^m$, compactly supported in $V\times\mathbb{R}^m$. Assume that $g$ is $\bar\partial$-flat with respect to $z$ on $ K\times I_\eta^m$. We then have, for any $(x,t)\in I_\delta\times I_\eta^m$,
\begin{equation*}
g(x,t)=P(x,t)q(x,t)+\sum_{j=0}^{d-1}r_j(t)x^j
\end{equation*}
where $q$ and $r_0,\ldots,r_{d-1}$ are given by integral representations
\begin{equation}\label{rep1}
q(x,t)= 
\int_{V\setminus K}
\frac{\partial g}{\partial\bar\zeta}(\zeta,t)\mathcal{K}(x,\zeta,t)d\lambda(\zeta)
\end{equation}
and
\begin{equation}\label{rep2}
r_j(t)= 
\int_{V\setminus K}
\frac{\partial g}{\partial\bar\zeta}(\zeta,t)\mathcal{K}_j(\zeta,t)d\lambda(\zeta)\textrm{ for } j=0,\ldots, d-1,
\end{equation}
with respective kernels
\begin{equation}\label{defK}
\mathcal{K}(x,\zeta,t)=\frac{1}{\pi^2}\int_{V\setminus K}\frac{\partial\varphi}{\partial\bar z}(z,\zeta)\frac{1}{P(z,t)(\zeta-z)(z-x)}d\lambda(z)
\end{equation} 
and
\begin{equation}\label{defKj}
\mathcal{K}_j(\zeta,t)=\frac{1}{\pi^2}\int_{V\setminus K}\frac{\partial\varphi}{\partial\bar z}(z,\zeta)\frac{S_j(z,t)}{P(z,t)(\zeta-z)}d\lambda(z).
\end{equation}
Here, $\varphi$ is the truncation function from Lemma \ref{lemphi}.
\end{prop}

\subsection{Estimates and the main result}
Regularity properties and estimates for the kernels $\mathcal{K}$ and $\mathcal{K}_0,\ldots,\mathcal{K}_{d-1}$ are established in the next lemma.  

\begin{lem}\label{estimkern}
 For a sufficiently small $\varepsilon \in ]0,\eta[$, the functions $\mathcal{K}$ and $\mathcal{K}_j$ are $\mathcal{C}^\infty$ on $I_\delta\times (V\setminus K)\times I_\varepsilon^m$ and $ (V\setminus K)\times I_\varepsilon^m$, respectively. Moreover, there are constants $C\geq 0$ and $\nu\geq 1$ such that, for any integer $k\geq 0$, any multi-index $R\in \mathbb{N}^n$of length $r$, and any $(x,\zeta,t)\in I_\delta\times (V\setminus K)\times I_\varepsilon^m$, we have the estimates
\begin{equation*}
\vert D_x^kD_t^R\mathcal{K}(x,\zeta,t)\vert\leq C^{k+r+1}k!\,r!\, d(\zeta,K)^{-(k+\sigma r+\nu+3)}
\end{equation*}
and, for $j=0,\ldots,d-1$,
\begin{equation*}
\vert D_t^R\mathcal{K}_j(\zeta,t)\vert\leq C^{r+1}r!\, d(\zeta,K)^{-(\sigma r+\nu+2)}.
\end{equation*}
\end{lem}
\begin{proof} Properties \eqref{phi1} and \eqref{phi2} imply that the integration in \eqref{defK} and \eqref{defKj} is actually performed on the set $ \{z\in V\setminus K : \frac{1}{64}d(\zeta,K)\leq d(z,K)\leq \frac{3}{4}d(\zeta,K)\}$. For any $z$ in that set, we have
\begin{equation}\label{Zz}
\vert \zeta-z\vert\geq d(\zeta,K)-d(z,K)\geq \frac{1}{4}d(\zeta,K)
\end{equation}
and
\begin{equation}\label{zx}
\vert z-x\vert\geq d(z,K)\geq \frac{1}{64}d(\zeta,K).
\end{equation}
Since we also have $d(z,\Gamma)\geq d(z,K)$, Proposition \ref{Cauchy} then yields $\varepsilon\in ]0,\eta[$, $\nu\geq 1$ and $C_1\geq 0$ such that we have
\begin{equation}\label{estiminvPK}
\left\vert D_t^R \left(\frac{1}{P(z,t)}\right)\right\vert\leq C_1^{r+1}r!\, d(\zeta,K)^{-(\sigma r+\nu)}
\end{equation}
for any such $z$, any multi-index $R\in \mathbb{N}^n$ and any $t\in I_\varepsilon^m$. From 
\eqref{phi3}, \eqref{Zz}, \eqref{zx} and \eqref{estiminvPK}, and from the analyticity of the coefficients of the polynomial $S_j$, it is then easy to derive the $\mathcal{C}^\infty$ regularity of the kernels $\mathcal{K}$ and $\mathcal{K}_j$, as well as the estimates on their derivatives.  
\end{proof}

We are now ready to state the main result. 

\begin{thm}\label{main} 
Let $M$ be a strongly regular sequence and let $f$ be a germ of function of class $\mathcal{C}_M$ at the origin in $\mathbb{R}\times \mathbb{R}^m$. Then there are germs of functions $q$ and $r_0,\ldots,r_{d-1}$ of class $\mathcal{C}_{M^\sigma}$ at the origin in $\mathbb{R}\times \mathbb{R}^m$ and $\mathbb{R}^m$ respectively, such that
\begin{equation}\label{divfinal}
f(x,t)=P(x,t)q(x,t)+\sum_{j=0}^{d-1}r_j(t)x^j.
\end{equation}
\end{thm} 
\begin{proof} We can assume that $f$ is of class $\mathcal{C}_M$ in a neighborhood of $I_\delta\times I_\eta^m$. With $f$ we associate a function $g$ as in Proposition \ref{mainextprop}, and we use Proposition \ref{divisionCC} to obtain a division formula for $g$, which in turn yields \eqref{divfinal}, where the functions $q$ and  $r_0,\ldots,r_{d-1}$ are given by the integral representations \eqref{rep1} and \eqref{rep2}. We proceed to establish the $\mathcal{C}_{M^\sigma}$ regularity of these functions. For any $(x,\zeta,t)\in I_\delta\times (V\setminus K)\times I_\eta^m$, any integer $k\geq 0$ and any multi-index $L\in \mathbb{N}^n$, the Leibniz formula yields
\begin{equation}\label{Leibniz}
D_x^kD_t^L\left(\frac{\partial g}{\partial\bar\zeta}(\zeta,t)\mathcal{K}(x,\zeta,t)\right)=\sum_{Q+R=L}\frac{L!}{Q!R!}\left(D_t^Q\frac{\partial g}{\partial\bar\zeta}(\zeta,t)\right)\left(D_x^kD_t^R\mathcal{K}(x,\zeta,t)\right).
\end{equation}
Denote by $\hat\zeta$ be a point in $K$ such that $ \vert \zeta-\hat\zeta\vert=d(\zeta,K)$. Since $ \partial g/\partial\bar\zeta$ is of class $\mathcal{C}_M$ in $\mathbb{C}\times\mathbb{R}^m$ and vanishes together with all its derivatives on $K\times I_\eta^m$, an application of the Taylor formula between $(\zeta,t)$ and $(\hat\zeta,t)$ shows that for a suitable constant $C_1\geq 0$, we have  
\begin{equation}\label{Taylor}
\left\vert D_t^Q\frac{\partial g}{\partial\bar\zeta}(\zeta,t)\right\vert\leq C_1^{p+q+1}M_{p+q}d(\zeta,K)^p\quad\textrm{for any } (p,Q)\in \mathbb{N}\times\mathbb{N}^m. 
\end{equation}
Given $(Q,R)\in \mathbb{N}^m\times\mathbb{N}^m $ such that $Q+R=L$, we set 
$ p=\lfloor k+\sigma r +\nu+3 \rfloor +1 $. Since $\sigma\geq 1$, we have $ p+q \leq \lfloor k+\sigma l\rfloor+ \lfloor\nu\rfloor+4$ which, together with \eqref{modg}, implies $M_{p+q}\leq C_2^{k+ l+1}M_{\lfloor k+\sigma l\rfloor}$ for a suitable $C_2\geq 0$. We also have $k!\, r!\leq (k+r)!\leq (k+l)!$. Combining \eqref{Taylor} with Lemma \ref{estimkern}, we therefore see that for $(x,\zeta,t)\in I_\delta\times (V\setminus K)\times I_\varepsilon^m$, each term of the summation in \eqref{Leibniz} satisfies
\begin{equation*}
\left\vert\left(D_t^Q\frac{\partial g}{\partial\bar\zeta}(\zeta,t)\right)\left(D_x^kD_t^R\mathcal{K}(x,\zeta,t)\right)\right\vert\leq C_3^{k+l+1}(k+l)!\, M_{\lfloor k+\sigma l\rfloor}
\end{equation*}
for a suitable constant $C_3\geq 0$. Since $ \sum_{Q+R=L}\frac{L!}{Q!R!}=2^l$, we derive 
\begin{equation*}
\left\vert D_x^kD_t^L\left(\frac{\partial g}{\partial\bar\zeta}(\zeta,t)\mathcal{K}(x,\zeta,t)\right)\right\vert\leq (2C_3)^{k+l+1}(k+l)!\, M_{\lfloor k+\sigma l\rfloor}
\end{equation*}
for any $(x,\zeta,t)\in I_\delta\times (V\setminus K)\times I_\varepsilon^m$.  
Together with the integral representation \eqref{rep1}, this yields an estimate of the form
\begin{equation}\label{aniso}
\left\vert D_x^kD_t^L q(x,t)\right\vert\leq C_4^{k+l+1}(k+l)!M_{\lfloor k+\sigma l \rfloor}
\end{equation} 
for any $(x,t)\in I_\delta\times I_\varepsilon^m$. 
From \eqref{equivseq} and \eqref{aniso}, we conclude that $q$ defines a germ of class $\mathcal{C}_{M^\sigma}$ at the origin in $\mathbb{R}\times\mathbb{R}^m$. The proof that $r_0,\ldots,r_{d-1}$ are germs of class $\mathcal{C}_{M^\sigma}$ at the origin in $\mathbb{R}^m$ goes exactly along the same lines. 
\end{proof}

\begin{rem} For the quotient $q$, we have actually proved a slightly more precise result: \eqref{aniso} indeed shows that $(x,t)\mapsto q(x,t)$ is $\mathcal{C}_{M}$ with respect to $x$ and $\mathcal{C}_{M^\sigma}$ with respect to $t$. 
\end{rem}

\begin{rem} We stress that, due to the non-quasianalytic context, the choice for the quotient and the remainder in the division formula is not unique, and only a suitable choice of $q$ and $r_0,\ldots,r_{d-1}$ may provide sharp estimates. For example, given a real $\alpha>0$, consider the function $h$ defined by $h(t)=e^{-\vert t\vert^{-1/\alpha}}$ for $t\neq 0$, and $h(0)=0$. It is well known that $h$ belongs to the Gevrey class $\mathcal{G}^{1+\alpha}$. For $m=1$ and $d=2$, set $P(x,t)=x^2+t^4$ and $f(x,t)=h(t)$. We can  write 
\eqref{divfinal} with the obvious choice $q=0$, $r_1=0$, $r_0=h$, which provides a quotient and a remainder in the same class as $f$. But \eqref{divfinal} also holds with $r_0=r_1=0$ and $q(x,t)=\frac{h(t)}{P(x,t)}$ for $(x,t)\neq (0,0)$, $q(0,0)=0$. In this case, it is known that $q$ belongs to the Gevrey class $\mathcal{G}^{1+2\alpha}$, but not to any smaller class \cite{Th1b}.      
\end{rem}

\subsection{Discussion of examples} 

\begin{exam}\label{above}
Given an integer $p\geq 1$, set $m=1$ and $P(x,t)=x^2+t^{2p}$. It has been shown in Section \ref{examp} that we have $\sigma=1$, hence the division formula \eqref{divfinal} holds with the quotient and remainder belonging to the same class $\mathcal{C}_M$ as $f$. A particular case of this result already appeared as a remark in Section 9 of \cite{CC2}, where it was obtained by an \emph{ad hoc} argument applicable when $f$ depends only on $x$. Our result provides a geometric interpretation of this remark. 
\end{exam}

\begin{exam}
Given an integer $d\geq 2$, set $m=1$ and $P(x,t)=x^d-t^2$. It has been shown in Section \ref{examp} that we have $\sigma=d/2$. We shall now prove, as announced in the Introduction, that the corresponding $\mathcal{C}_{M^\sigma}$ estimate in the division formula of Theorem \ref{main} is optimal. Writing \eqref{divfinal} for a $\mathcal{C}_M$ function $f: x\mapsto f(x)$ that depends only on $x$, and setting $ (x,t)=(y^2,y^d)$ for $y$ in a neighborhood of $0$, we get
\begin{equation}\label{identif}
f(y^2)= \sum_{j=0}^{d-1}r_j(y^d)y^{2j}.
\end{equation}
Assume that $r_0,\ldots,r_{d-1}$ are of class $\mathcal{C}_N$ for some sequence $N$ satisfying \eqref{norm}, \eqref{logc} and \eqref{modg}. Then, for each integer $k\geq 0$, an identification of the terms of degree $2dk$ in the formal Taylor series at $0$ of both members of \eqref{identif} yields
\begin{equation*}
\frac{1}{(dk)!}\vert f^{(dk)}(0)\vert\leq C^{k+1} N_{2k}
\end{equation*} for some suitable constant $C\geq 0$. It is known that one may choose $f$ such that $\vert f^{(j)}(0)\vert \geq j!M_j$ for any $j\geq 0$, see for instance Theorem 1 in \cite{Th2}. We then get $M_{dk}\leq C^{k+1} N_{2k}$ for any $k$, which, thanks to \eqref{equivseq}, implies $N_k\geq C_1^{k+1}M_k^{d/2}$ for some $C_1>0$. Hence, the class $\mathcal{C}_N$ cannot be smaller than $\mathcal{C}_{M^\sigma}$.  
\end{exam}

\begin{exam}
Set $m=1$ and $P(x,t)=x^2-2tx+t^2+t^4$. We have seen in Section \ref{examp} that the geometric assumptions \ref{geomgamma} are not satisfied, hence our scheme of proof cannot be applied. More specifically, since $\Gamma$ and the real-axis are $2$-regularly separated, the argument of glueing of jets used in the proof of Proposition 
\ref{mainextprop} would only provide a $\bar\partial$-flat extension in the wider class $\mathcal{C}_{M^2}$. However, let $f$ be a germ of class $\mathcal{C}_M$ at the origin in $\mathbb{R}\times\mathbb{R}$, and set $F(x,t)=f(x+t,t)$. Taking Example \ref{above} into account, we may write $F(x,t)=(x^2+t^4)Q(x,t)+xR_1(t)+R_0(t)$ where $Q$, $R_0$ and $R_1$ are of class $\mathcal{C}_M$. Thus, we get 
\begin{equation*}
f(x,t)=F(x-t,t)=P(x,t)q(x,t)+xr_1(t)+r_0(t)
\end{equation*}
with $q(x,t)=Q(x-t,t)$, $r_1(t)=R_1(t)$ and $r_0(t)=R_0(t)-tR_1(t)$. In particular, $q$, $r_0$ and $r_1$ are of class $\mathcal{C}_M$. This \emph{ad hoc} argument shows that Weierstrass division by $P$ holds with the quotient and the remainder belonging to the same class $\mathcal{C}_M$, even though the geometric assumptions \ref{geomgamma} are not satisfied. This suggests that it might be possible to get rid of some of those assumptions and obtain a more general result.
\end{exam}

\end{document}